\documentclass[12pt,twoside]{amsart}
\usepackage{amssymb,amsmath,amsthm, amscd, enumerate, mathrsfs}
\usepackage{graphicx, hhline}
\usepackage[all]{xy}
\usepackage{color}

\title[Boundedness of B-pluricanonical representations]
{Boundedness of log pluricanonical representations of log Calabi--Yau pairs in dimension 2}
\author{Chen Jiang and Haidong Liu}
\date{2020/2/21, version 0.04}
\address{Fudan University, Shanghai Center for Mathematical Sciences, 
Shanghai, 200438, China}
\email{chenjiang@fudan.edu.cn}
\address{Peking University, Beijing International Center for Mathematical Research, 
Beijing, 100871, China}
\email{hdliu@bicmr.pku.edu.cn}

\DeclareMathOperator{\Supp}{Supp}
\DeclareMathOperator{\Bir}{Bir}
\DeclareMathOperator{\Aut}{Aut}
\DeclareMathOperator{\id}{id}
\DeclareMathOperator{\PA}{PA}
\DeclareMathOperator{\A}{A}
\DeclareMathOperator{\Spec}{Spec}

\newtheorem{thm}{Theorem}[section]
\newtheorem{lem}[thm]{Lemma}
\newtheorem{prop}[thm]{Proposition}
\newtheorem{conj}[thm]{Conjecture}
\newtheorem{cor}[thm]{Corollary}
\newtheorem*{conja}{Conjecture A}
\newtheorem*{conjb}{Conjecture B}
\newtheorem*{conjb'}{Conjecture B'}

\newtheorem*{conjc}{Conjecture C}

\theoremstyle{definition}

\newtheorem{defn}[thm]{Definition}
\newtheorem{rem}[thm]{Remark}
\newtheorem*{ack}{Acknowledgments}

\makeatletter
 
 \@addtoreset{equation}{section}
\makeatother

\begin{document}

\begin{abstract}
We show the
 boundedness of B-pluricanonical representations of lc log Calabi--Yau pairs in dimension $2$. 
As applications, we prove the boundedness of indices of slc log Calabi--Yau pairs up to dimension $3$ and that of non-klt lc log Calabi--Yau pairs in dimension $4$.
\end{abstract} 

\maketitle 

\tableofcontents

\section{Introduction}\label{sec1}

In the framework of Fujino \cite{fujino-ab}, the finiteness of B-pluricanonical representations
(or equally, log pluricanonical representations, see Definition \ref{b-rep}) plays an important role in the study of the abundance conjecture in the minimal model program. The finiteness of B-pluricanonical representations was investigated by Fujino \cite{fujino-ab} and Gongyo \cite{gongyo} after the work of Nakamura--Ueno \cite{nu} and Deligne \cite[Section 14]{ueno}, and proved in its full generality by Fujino--Gongyo \cite{fujino-gongyo}.

In this paper, we are interested in the B-pluricanonical representations
of log Calabi--Yau pairs. Log Calabi--Yau pairs form an important class in the minimal model program. It is expected that log Calabi--Yau pairs should satisfy certain boundedness properties, see \cite{alexeev, alexeev-mori, DCS, rccy3, Birkar18, jiang, xu2} for related works. Therefore, it is natural to consider the following conjecture on boundedness of B-pluricanonical representations of log Calabi--Yau pairs (cf. \cite[Conjecture 3.2]{fujino-index}, \cite[Conjecture 1.9]{xu2}, \cite[Conjecture 8.3]{fmx19}).

\begin{conj}\label{main-conj}
Let $m, d$ be two positive integers. 
Then there exists a positive integer $N$ depending only on $m, d$ satisfying the following property:
if $(X, \Delta)$ is a projective connected
lc pair of dimension $d$ such that $m(K_X+\Delta)\sim 0$, then
$|\rho_{km}(\Bir(X,\Delta))|\leq N$ for any positive integer $k$. 
\end{conj}

\begin{rem}\label{remark1}
If $(X, \Delta)$ is a projective connected
lc pair of dimension $d$ such that $m(K_X+\Delta)\sim 0$, then $H^0(X, km(K_X+\Delta))\simeq \mathbb{C}$ for any positive integer $k$. 
Hence $\rho_{km}(g)\in \mathbb{C}^*$ and $\rho_{km}(g)=\rho_{m}(g)^k$ for any $g\in \Bir(X, \Delta)$. 
So $$|\rho_{km}(\Bir(X,\Delta))|\leq |\rho_{m}(\Bir(X,\Delta))|\leq k |\rho_{km}(\Bir(X,\Delta))|.$$
Therefore, in Conjecture \ref{main-conj}, it suffices to consider the case $k=1$.
\end{rem}

\begin{rem}\label{rem-conj-1}
In Conjecture \ref{main-conj}, the {assumptions} that $(X, \Delta)$ is connected and lc are necessary.
In fact, it is easy to see that $|\rho_m(\Bir(X,\Delta))|$ could not be bounded uniformly unless the number of irreducible components of $X$ is bounded. For example, if
$X$ is a cycle of smooth rational curves or a disjoint union of several copies of an elliptic {curve}, then $K_X\sim 0$ but $|\rho_1(\Bir(X,0))|$ depends on the number of irreducible components of $X$ (a rotation of irreducible components of $X$ gives 
a B-pluricanonical representation). 
\end{rem}

Conjecture \ref{main-conj} can be easily proved in dimension $1$, see \cite[Theorem 3.3]{fujino-ab}, \cite[Page 18]{xu1}, or \cite[Proposition 8.4]{fmx19}. But it was still open even in dimension 2. As the main result of this paper, we give an affirmative answer to Conjecture \ref{main-conj} in dimension 2.
 
\begin{thm}[{=Theorem \ref{thm-klt} + Theorem \ref{thm-dlt}}]\label{main-thm}
Let $m$ be a positive integer. 
Then there exists a positive integer $N$ depending only on $m$ satisfying the following property:
if $(S, B)$ is a projective connected
lc pair of dimension $2$ such that $m(K_S+B)\sim 0$, then
$|\rho_{km}(\Bir(S,B))|\leq N$ for any positive integer $k$. 
\end{thm}

In the framework of Fujino \cite{fujino-ab}, it is known that Conjecture \ref{main-conj} is closely related to the following index conjecture for log Calabi--Yau pairs (cf. \cite[Conjecture 1.3]{xu2}):

\begin{conj}[Index conjecture]\label{index conj slc}
Let $I$ be a finite set in $[0,1]\cap \mathbb{Q}$ and $d$ a positive integer. Then there exists 
a positive integer $m$ depending only on $I, d$ satisfying the following property:
if $(X,\Delta)$ is a projective slc pair
of dimension $d$ such that the coefficients of $\Delta$ are in $I$
and $K_X+\Delta\sim_{\mathbb Q} 0$, 
then $ m(K_X+\Delta) \sim 0$.
\end{conj}

Recently this conjecture was studied by the first author \cite{jiang} and Xu \cite{xu1, xu2}. It was proved in dimension 2 \cite{xu1}, and in dimension 3 for lc pairs \cite{jiang, xu2}. 
As applications of Theorem \ref{main-thm}, we prove the boundedness of indices of slc log Calabi--Yau pairs up to dimension $3$.

\begin{cor}\label{ind-threefold}
Conjecture \ref{index conj slc} holds in dimension $\leq 3$. To be more precise, let $I$ be a finite set in $[0,1]\cap \mathbb{Q}$. Then there exists 
a positive integer $m$ depending only on $I$ satisfying the following property:
if $(X,\Delta)$ is a projective slc pair
of dimension at most $3$ such that the coefficients of $\Delta$ are in $I$
and $K_X+\Delta\sim_{\mathbb Q} 0$, 
then $ m(K_X+\Delta) \sim 0$.
\end{cor}

\begin{cor}\label{ind-fourfold}
Conjecture \ref{index conj slc} holds for connected non-klt lc pairs in dimension $4$. To be more precise, let $I$ be a finite set in $[0,1]\cap \mathbb{Q}$. Then there exists 
a positive integer $m$ depending only on $I$ satisfying the following property:
if $(X,\Delta)$ is a projective connected non-klt lc pair
of dimension $4$ such that the coefficients of $\Delta$ are in $I$
and $K_X+\Delta\sim_{\mathbb Q} 0$, 
then $ m(K_X+\Delta) \sim 0$.
\end{cor}

\begin{ack}
The authors would like to thank Osamu Fujino and Jingjun Han for helpful discussions. 
The idea of this work was carried out during the conference ``2019 Suzhou AG Young Forum" at Soochow University and the authors are grateful to Yi Gu and Cheng Gong for hospitality and support.
The first author was supported by Start-up Grant No. SXH1414010.
\end{ack}

\section{Preliminaries}\label{pre}

\subsection{Notation and conventions}
We work over the complex number field $\mathbb C$ throughout 
this paper. We freely use the basic 
notation of the minimal model program in 
\cite{kollar-mori}.
In this paper, we consider only $\mathbb Q$-divisors instead of $\mathbb R$-divisors. 
A {\em scheme} is always assumed to be separated and of finite type over $\mathbb{C}$.
The {\em dimension} of a scheme is the pure dimension of that scheme, that is, when we consider the dimension of a scheme $X$, $X$ is always assumed to be of pure dimension.
 A {\em variety} is a reduced and irreducible scheme. A {\em curve} (resp. {\em surface}) is a variety of dimension 1 (resp. 2).
A normal scheme consists of the disjoint union of irreducible normal schemes.

Let $D$ be a $\mathbb Q$-divisor on a normal scheme $X$, that is, 
$D$ is a finite formal sum $\sum _i d_i D_i$ where 
$d_i\in \mathbb{Q}$ and $\{D_i\}_i$ are distinct prime divisors 
on $X$. 
We put 
$$
D^{<1}=\sum _{d_i<1}d_iD_i, \quad 
D^{\leq 1}=\sum _{d_i\leq 1}d_i D_i, \quad 
\text{and} \quad D^{=1}=\sum _{d_i=1}D_i. 
$$
We also put
$
\lfloor D\rfloor =\sum _i \lfloor d_i \rfloor D_i$
where $\lfloor d_i\rfloor$ is the integer defined by 
 {$d_i-1<\lfloor d_i\rfloor \leq d_i$},
and put $\{D\}=D-\lfloor D\rfloor$.

\subsection{Singularities of pairs}
A {\em sub-pair} $(X, \Delta)$ consists of a normal 
scheme $X$ and a $\mathbb Q$-divisor 
$\Delta$ on $X$ such that $K_X+\Delta$ is $\mathbb Q$-Cartier. 
Let $f:Y\to X$ be a projective 
birational morphism from a normal scheme $Y$. 
Then we can write 
$$
K_Y=f^*(K_X+\Delta)+\sum _E a(E, X, \Delta)E
$$ 
with 
$$f_*\left(\underset{E}\sum a(E, X, \Delta)E\right)=-\Delta, 
$$ 
where $E$ runs over prime divisors on $Y$. 
We call $a(E, X, \Delta)$ the {\em{discrepancy}} of $E$ with 
respect to $(X, \Delta)$. 
Note that we can define the discrepancy $a(E, X, \Delta)$ for 
any prime divisor $E$ over $X$ by taking a suitable 
resolution of singularities of $X$. 
If $a(E, X, \Delta)\geq -1$ (resp. $ a(E, X, \Delta)>-1$) for 
every prime divisor $E$ over $X$, 
then $(X, \Delta)$ is called {\em{sub log canonical}} ({\em{sub-lc}} for short) 
or {\em{sub Kawamata log terminal}} ({\em{sub-klt}} for short) respectively.

If $\Delta$ is effective, then a sub-pair $(X, \Delta)$ is called a {\em pair}, and
$(X, \Delta)$ is 
called {\em{lc}} (resp. {\em{klt}}) if it is sub-lc (resp. sub-klt). 
A {\em{divisorial log terminal}} ({\em{dlt}} for short)
pair is a limit of klt pairs in the sense of \cite[Proposition 2.43]{kollar-mori}
(see \cite[Definition 2.37 and Proposition 2.40]{kollar-mori} for precise definitions).

Let $(X, \Delta)$ be a sub-lc pair. 
If there exist a projective birational morphism 
$f:Y\to X$ from a normal scheme $Y$ and a prime divisor $E$ on $Y$ 
with $a(E, X, \Delta)=-1$, then $f(E)$ is called an {\em{lc center}} of 
$(X, \Delta)$.

Let $X$ be a reduced scheme of pure dimension which satisfies Serre's $S_2$ condition and 
is normal crossing in codimension one. 
Let $\Delta$ be an effective $\mathbb Q$-divisor 
on $X$ such that no irreducible component of $\Supp \Delta$ is 
contained in the singular locus of $X$ and $K_X+\Delta$ is $\mathbb Q$-Cartier. 
We say that $(X, \Delta)$ is a {\em{semi log canonical}} 
({\em{slc}} for short) pair if 
$(X^\nu, \Delta_{X^\nu})$ is log canonical, 
where $\nu:X^\nu\to X$ is the normalization of 
$X$ and $K_{X^\nu}+\Delta_{X^\nu}=\nu^*(K_X+\Delta)$, 
that is, $\Delta_{X^\nu}$ is the sum of the inverse image of $\Delta$ 
and the conductor of $X$. 
We say that $(X, \Delta)$ is a {\em{semi divisorial log terminal}} 
({\em{sdlt}} for short) pair if 
$(X^\nu, \Delta_{X^\nu})$ is dlt, and every irreducible component of $X$ is normal. 
Note that an sdlt pair is naturally an slc pair.
For more details of slc and sdlt pairs, see \cite{fujino-ab, fujino-slc, kollar, kollar92}. 

\subsection{Log Calabi--Yau pairs}
A pair $(X, \Delta)$ is called a
\emph{log Calabi--Yau pair} if $X$ is projective and $K_X+B\sim_\mathbb{Q} 0$, and the {\it (global) index} of $(X, \Delta)$ is the minimal positive integer $m$ such that $m(K_X+\Delta)\sim 0$.

\subsection{B-birational maps}
We recall basic knowledge on B-birational maps and B-pluricanonical representations introduced by \cite{fujino-ab}.
For more details, see \cite{fujino-ab, fujino-gongyo, gongyo} and the
references therein.

\begin{defn}[{\cite[Definition 1.5]{fujino-ab} or \cite[Definition 2.11]{fujino-gongyo}}]\label{b-bir}
Let $(X,\Delta)$ and $(Y,\Delta_Y)$ be two sub-pairs.
 A proper birational map
$f:(X,\Delta) \dashrightarrow (Y,\Delta_Y)$ is called \emph{B-birational} if there exists 
a common resolution 
$$
\xymatrix{
&W\ar[ld]_{\alpha}\ar[rd]^{\beta}\\ 
(X,\Delta) \ar@{-->}[rr]^{f} & & (Y,\Delta_Y)
}
$$
such that $\alpha^*(K_X+\Delta)=\beta^*(K_Y+\Delta_Y)$.
In this case, we say that $(Y, \Delta_Y)$ is a \emph{B-birational model} or a \emph{crepant model} of $(X, \Delta)$.
Let 
$$
\Bir(X, \Delta)=\{f| f: (X,\Delta)
 \dashrightarrow (X,\Delta) \text{ is B-birational}\}.
$$
Then $\Bir(X, \Delta)$ has a natural group structure under compositions of maps.
\end{defn}

\begin{defn}[{\cite[Definition 3.1]{fujino-ab} or \cite[Definition 2.14]{fujino-gongyo}}]\label{b-rep}
Let $(X, \Delta)$ be a sub-pair. Fix a positive integer $m$ such that $m(K_X+\Delta)$
is Cartier. Then a B-birational map $f: (X,\Delta) \dashrightarrow (X,\Delta)$ 
naturally induces a linear automorphism of $H^0(X, m(K_X+\Delta))$. This gives a 
group homomorphism
$$
\rho_m: \Bir(X, \Delta)\to \Aut_{\mathbb C} (H^0(X, m(K_X+\Delta))).
$$
The homomorphism $\rho_m$ is called a \emph{B-pluricanonical representation} 
or \emph{log pluricanonical representation} for $(X, \Delta)$.
We sometimes denote $\rho_m(g)$ by $g^*$ for $g\in \Bir(X, \Delta)$ if there is no 
danger of confusion.
\end{defn}

\begin{rem}\label{rem bir=bir}
As explained in \cite[Remark 2.12]{fujino-gongyo}, if $f:(X,\Delta) \dashrightarrow (Y,\Delta_Y)$ is a {B-birational} map, then there is a group isomorphism $\Bir(X, \Delta)\simeq \Bir(Y, \Delta_Y)$ given by $g\mapsto f\circ g\circ f^{-1}$. 
\end{rem}

For B-pluricanonical representation, we have the following finiteness theorem proved by Fujino--Gongyo \cite{fujino-gongyo} (later Hacon--Xu \cite{hx16}  gave a different proof for a weaker statement). It is 
a log version of Nakamura--Ueno \cite{nu} 
and Deligne--Ueno's finiteness theorem of pluricanonical representations \cite[Theorem 14.10]{ueno}.

\begin{thm}[{\cite[Theorem 1.1]{fujino-gongyo}}]
Let $(X,\Delta)$ be a projective lc pair. Assume that $m(K_X+\Delta)$
is Cartier and $K_X+\Delta$ is semi-ample. Then
$\rho_m(\Bir(X, \Delta))$ is a finite group.
\end{thm}

\subsection{Cyclic coverings}We recall the construction of $m$-fold cyclic coverings.

For simplicity, 
we assume that $X$ is a smooth variety, $\Delta$ is a $\mathbb Q$-divisor on $X$ with simple normal crossing support
such that $K_X+\Delta\sim_\mathbb{Q} 0$. Take $m\in \mathbb Z_{> 0}$ to be the minimal one such that $m(K_X+\Delta)\sim 0$. Then
there is an {\em $m$-fold cyclic covering}
corresponding to the effective divisor
$m\{\Delta\}\sim m(-K_{X}-\lfloor\Delta \rfloor)$ given by
$$
\mu: {\tilde{X}}=\Spec\left( \bigoplus_{i=0}^{m-1}\mathcal L^{-i}
(\lfloor i\{\Delta \} \rfloor)\right) \to X
$$ 
where $\mathcal L=\mathcal O_{X}(-K_{X}-\lfloor\Delta \rfloor)$
 (cf. 
\cite[2.3]{kollar}).
There is also an alternative description of above $m$-fold cyclic covering as
$$
\mu: {\tilde{X}}=\Spec\left( \bigoplus_{i=0}^{m-1}\mathcal O_{X}(\lfloor i(K_{X}+\Delta)\rfloor)\right) \to X
$$ 
as in \cite[Section 6]{fujino-slc-trivial} by $m(K_{X}+\Delta) \sim 0$. 
More precisely, fix a non-zero section $\omega\in H^0(X, m(K_{X}+\Delta))$, the $\mathcal O_{X}$-algebra structure of 
$\bigoplus_{i=0}^{m-1}\mathcal O_{X}(\lfloor i(K_{X}+\Delta)\rfloor)$
is given by the natural multiplication
$$
\mathcal O_{X}(\lfloor i(K_{X}+\Delta)\rfloor) \otimes
\mathcal O_{X}(\lfloor j(K_{X}+\Delta)\rfloor)
\to
\mathcal O_{X}(\lfloor (i+j)(K_{X}+\Delta)\rfloor)
$$
if $i+j<m$, and by the multiplication
\begin{align*}
\mathcal O_{X}(\lfloor i(K_{X}+\Delta)\rfloor) \otimes
\mathcal O_{X}(\lfloor j(K_{X}+\Delta)\rfloor) {}&\to \mathcal O_{X}(\lfloor (i+j)(K_{X}+\Delta)\rfloor)\\
 {}&\stackrel{\times \omega^{-1}}\longrightarrow
\mathcal O_{X}(\lfloor (i+j-m)(K_{X}+\Delta)\rfloor)
\end{align*}
if $i+j\geq m$.
Since we have
$$
\mathcal L^{-i}(\lfloor i\{\Delta\} \rfloor)=
\mathcal O_{X}(iK_{X}+i\lfloor\Delta\rfloor+\lfloor i\{\Delta\} \rfloor)
=\mathcal O_{X}(\lfloor i(K_{X}+\Delta)\rfloor),
$$
these two descriptions are indeed isomorphic. 
 Note that ${\tilde{X}}$ is not necessarily smooth, but it is normal and irreducible by the minimality of $m$. Also note that $\mu$ is \'etale outside $\Supp \{\Delta\}$. Since the construction of ${\tilde{X}}$ depends on the choice of $\omega$, usually we denote this covering by $
\mu: {\tilde{X}}_\omega \to X.
$
By the construction, there exists a $\mathbb{Q}$-divisor $\Delta_{{\tilde{X}}_\omega}$ on ${\tilde{X}}_\omega$ such that $K_{{\tilde{X}}_\omega}+\Delta_{{\tilde{X}}_\omega}=\mu^*(K_X+\Delta)$ and $K_{{\tilde{X}}_\omega}+\Delta_{{\tilde{X}}_\omega}\sim 0$. 

In order to consider the lifting of B-birational maps after cyclic coverings, we have the following lemma.

\begin{lem}\label{comm-lem}Keep the above setting.
Let $\alpha: (W, \Delta_W)\to (X,\Delta)$ be a log resolution
such that $m(K_W+\Delta_W)=m\alpha^*(K_{X}+\Delta)\sim 0$.
Fix a non-zero section $\bar\omega \in H^0(W, m(K_W+\Delta_W))$.
Let $\mu: {\tilde{X}}_\omega\to X$ (resp. $\nu: {\tilde{W}}_{\bar\omega}\to W$) be the $m$-fold cyclic covering given by the 
section $\omega$ (resp. $\bar\omega$).
Then there exists a birational morphism $\tilde{\alpha}_{\omega, \bar\omega}:{\tilde{W}}_{\bar \omega} \to {\tilde{X}}_{\omega}$ making the following diagram {commute}:
\begin{equation}\label{eq2.1}
\xymatrix{
{\tilde{W}}_{\bar \omega} \ar[d]_{\nu}\ar[r]^{\tilde{\alpha}_{\omega, \bar\omega}}& {\tilde{X}}_{\omega}\ar[d]^{\mu}\\ 
W \ar[r]_{\alpha} & X.
}
\end{equation}
\end{lem}

\begin{proof}
Since $ H^0(W, m(K_W+\Delta_W)) \simeq \mathbb C$ by assumption, we can 
take a $t\in \mathbb C^{*}$ such that $\bar\omega=t\alpha^*\omega$. Fix a primitive $m$-th root $\sqrt[m]{t}$ of $t$. 
We construct a morphism $\tilde{\alpha}_{\omega, \bar\omega}:{\tilde{W}}_{\bar \omega} \to {\tilde{X}}_{\omega}$ following the construction of
the $m$-fold cyclic covering. 
Note that there exists a natural isomorphism (see \cite[II.2.11]{Nakayama})
$$
\mathcal O_{X}(\lfloor i(K_{X}+\Delta)\rfloor)
\simeq \alpha_* \mathcal O_{W}(\lfloor i(K_W+\Delta_W)\rfloor)
$$
for every $i\geq 0$. 
We can consider the following ``twisted" isomorphism
$$
\mathcal O_{X}(\lfloor i(K_{X}+\Delta)\rfloor)\simeq \alpha_* \mathcal O_{W}(\lfloor i(K_W+\Delta_W)\rfloor)\stackrel{\times \sqrt[m]{t^i}}\longrightarrow \alpha_* \mathcal O_{W}(\lfloor i(K_W+\Delta_W)\rfloor),
$$
then it is easy to check that the induced isomorphism
$$
\bigoplus_{i=0}^{m-1}\mathcal O_{X}(\lfloor i(K_{X}+\Delta)\rfloor)\longrightarrow \bigoplus_{i=0}^{m-1} \alpha_* \mathcal O_{W}(\lfloor i(K_W+\Delta_W)\rfloor)
$$
is compatible with $\mathcal O_{X}$-algebra structures. 
So this isomorphism gives a birational morphism between coverings $\tilde{\alpha}_{\omega, \bar\omega}:{\tilde{W}}_{\bar \omega} \to {\tilde{X}}_{\omega}$.
\end{proof}

Then we can show that B-birational maps lift to cyclic coverings.

\begin{lem}\label{lift-lem}
Let $X$ be a smooth variety, $\Delta$ a $\mathbb Q$-divisor on $X$ with simple normal crossing support
such that $K_X+\Delta\sim_\mathbb{Q} 0$. Take $m\in \mathbb Z_{> 0}$ to be the minimal one such that $m(K_X+\Delta)\sim 0$. Let $\mu: {\tilde{X}}_\omega\to X$ be the $m$-fold cyclic covering given by a 
non-zero section $\omega \in H^0(X, m(K_{X}+\Delta))$.
Then
a B-birational {map} $g: (X,\Delta) \dashrightarrow (X,\Delta) $
can be lifted to a B-birational map 
$g': ({\tilde{X}}_\omega,\Delta_{{\tilde{X}}_\omega}) \dashrightarrow ({\tilde{X}}_\omega,\Delta_{{\tilde{X}}_\omega})$ commuting with $\mu$.
\end{lem}

\begin{proof}
Consider a common log resolution
$$
\xymatrix{
&(W, \Delta_W)\ar[ld]_{\alpha}\ar[rd]^{\beta}\\ 
(X,\Delta) \ar@{-->}[rr]^{g} & & (X,\Delta).
}
$$
Fix a non-zero section $\bar\omega \in H^0(W, m(K_W+\Delta_W))$.
Then by Lemma \ref{comm-lem}, we 
have the following commutative diagram
\begin{equation*}
\xymatrix{&&{\tilde{W}}_{\bar \omega} \ar[dll]_-{ \tilde{\alpha}_{\omega, \bar\omega} }
\ar[dd]^(.60){\nu}\ar[drr]^-{\tilde{\beta}_{\omega, \bar\omega}}&& \\ 
 {\tilde{X}}_\omega\ar[dd]_--{\mu} 
&&&& {\tilde{X}}_\omega\ar[dd]^-{\mu} \\
&& W\ar[dll]_-{\alpha}\ar[drr]^-{\beta} && \\ 
X \ar@{-->}[rrrr]_-g &&&& X
}
\end{equation*}
Note that $$ \tilde{\alpha}_{\omega, \bar\omega}^*(K_{ {\tilde{X}}_\omega}+\Delta_{ {\tilde{X}}_\omega})=\nu^*\alpha^*(K_X+\Delta)=\nu^*\beta^*(K_X+\Delta)=
 \tilde{\beta}_{\omega, \bar\omega}^*(K_{{\tilde{X}}_\omega}+\Delta_{{\tilde{X}}_\omega}).$$
Hence $ \tilde{\beta}_{\omega, \bar\omega}\circ \tilde{\alpha}_{\omega, \bar\omega}^{-1}: ({\tilde{X}}_\omega,\Delta_{{\tilde{X}}_\omega}) \dashrightarrow ({\tilde{X}}_\omega,\Delta_{{\tilde{X}}_\omega})$ is the required B-birational map.
\end{proof}

The following lemma
is a special case of \cite[Proposition 4.9]{gongyo} 
(see also \cite[Proposition 14.4]{ueno}).
The proof is essentially the same as that in \cite[Proposition 4.9]{gongyo}. Note that in \cite[Proposition 4.9]{gongyo}, cyclic coverings and liftings of B-birational maps are constructed locally and analytically, so here we modify the proof by the algebraic construction of coverings and liftings (Lemma \ref{lift-lem}).

\begin{lem}[{cf.  \cite[Proposition 4.9]{gongyo}, \cite[Remark 3.6]{fujino-gongyo}}]\label{cover-lem}
Let $(X,\Delta)$ be a projective sub-klt pair 
of dimension $d$ such that $X$ is smooth connected, $\Delta$ is with simple normal crossing support, 
and $K_X+\Delta\sim_\mathbb{Q} 0$. Take $m\in \mathbb Z_{> 0}$ to be the minimal one such that $m(K_X+\Delta)\sim 0$. Fix a non-zero section $\omega\in H^0(X, m(K_X+\Delta))$. 
Let $\mu: {\tilde{X}}_\omega\to X$ be the $m$-fold cyclic covering given by the 
section $\omega$. Take $\phi: V \to ({\tilde{X}}_\omega,\Delta_{{\tilde{X}}_\omega})$ to be any log resolution.
Take $N_V$ to be the least common multiple of all positive integers $k$ such that $\varphi(k)\leq b_d(V)$
where $b_d(V)$ is the $d$-th Betti number of $V$ 
and $\varphi$ is the Euler function. 
Then
for any B-birational map $g\in \Bir(X,\Delta)$, $(g^*)^{N_V}$ is the identity map on $H^0(X, m(K_X+\Delta)) \simeq \mathbb C$.
In particular, $|\rho_m(\Bir(X,\Delta))|\leq N_V$.
\end{lem}

\begin{proof}
Fix any B-birational map $g\in \Bir(X,\Delta)$.
Suppose that $g^*\omega=\lambda \omega$ for some $\lambda\in \mathbb C^*$. It suffices to show that $\lambda^{N_V}=1$.

We can view $\omega\in H^0(X, m(K_X+\Delta))$ as a non-zero meromorphic $m$-ple $d$-form. Then by the cyclic covering construction, there is a non-zero meromorphic $d$-form
$ \omega_V \in H^0(V, K_V + \Delta_V)$ on $V$
such that 
$$
( \omega_V)^m=\phi^*\mu^* \omega,
$$
where $K_V + \Delta_V=\phi^*(K_{{\tilde{X}}_\omega}+ \Delta_{{\tilde{X}}_\omega})$.
As $(X,\Delta)$ is sub-klt, $( V, \Delta_{ V})$ and $({\tilde{X}}_\omega, \Delta_{{\tilde{X}}_\omega})$
are sub-klt by \cite[Proposition 5.20]{kollar-mori}. By \cite[Lemmas 3.4]{fujino-gongyo}, 
$\omega$ is $L^{2/m}$-integrable.
Hence
 $ \omega_V|_{V\backslash\text{Supp}(\Delta_V)}$ is $L^2$-integrable and $ \omega_V$ is a holomorphic $d$-form by \cite[Proposition 16]{kawamata} or \cite[Lemma 3.3]{fujino-gongyo}. 
That is,
$ \omega_V \in 
H^0(V, K_V)\subset H^d(V, \mathbb Z)\otimes \mathbb{C}$.
By Lemma \ref{lift-lem}, there is a lifting $ g'\in \Bir({\tilde{X}}_\omega,\Delta_{{\tilde{X}}_\omega}) $ which naturally lifts to $ g_V\in \Bir(V, \Delta_V ) $ such that
$$
 g_V^*( \omega_V)^m= g_V^*\phi^*\mu^*(\omega)=
\phi^*\mu^* g^*\omega=\phi^*\mu^*( \lambda \omega)=\lambda( \omega_V)^m.
$$
As $H^0(V, K_V + \Delta_V)\simeq \mathbb{C}$, we can write $g_V^*( \omega_V)=\lambda'\omega_V$ for some $\lambda'\in \mathbb C^*$ with $(\lambda')^m=\lambda$.
By \cite[Proposition 14.4 and Theorem 14.10]{ueno}, we immediately see that $(\lambda')^{N_V}=1$. More precisely, by \cite[Theorem 14.10]{ueno}, $\lambda'$ is a root of unity; by \cite[Proposition 14.4]{ueno}, $\lambda'$ is an algebraic integer and the degree of the minimal polynomial of $\lambda'$ with coefficients in $\mathbb{Q}$ is bounded from above by $b_d(V)$, hence $(\lambda')^{N_V}=1$ by the definition of $N_V$.
Therefore, $(\lambda)^{N_V}=(\lambda')^{mN_V}=1$.
\end{proof}

\subsection{Bounded pairs}\label{sec.bdd}
A collection of projective varieties $ \mathcal{D}$ is
said to be \emph{bounded} if there exists a projective morphism 
$h\colon \mathcal{X}\rightarrow T$
between schemes of finite type such that
each $X\in \mathcal{D}$ is isomorphic to $\mathcal{X}_t$ 
for some closed point $t\in T$ where $\mathcal{X}_t=h^{-1}(t)$.

We say that a collection of projective connected log pairs $\mathcal{D}$ is \emph{log bounded} 
if there is a scheme $\mathcal{X}$, a 
reduced divisor $\mathcal{B}$ on {$\mathcal X$}, and a 
projective morphism $h\colon \mathcal{X}\to T$, where 
$T$ is of finite type and $\mathcal{B}$ does not contain 
any fiber, such that for every $(X,B)\in \mathcal{D}$, 
there is a closed point $t \in T$ and an isomorphism $f \colon \mathcal{X}_t \to X$ such that $\mathcal{B}_t:=\mathcal{B}|_{\mathcal{X}_t}$ 
coincides with the support of $f_*^{-1}B$.

Moreover, if $\mathcal{D}$ is a set of connected log Calabi--Yau pairs, then it is 
said to be {\it log bounded modulo B-birational contractions} if there exists another set $\mathcal{D}'$ of connected log Calabi--Yau pairs which is log bounded, and for each $(X, B)\in \mathcal{D}$, there exists $(X', B')\in \mathcal{D}$ and a B-birational map $g:(X, B)\dashrightarrow (X', B')$ which is a {\em contraction}, that is, $g$ does not extract any divisor.
Here we remark that the concept of log boundedness modulo B-birational contractions is a weaker version of ``log boundedness modulo flops" introduced in \cite{rccy3, jiang}, in which $g$ is assumed to be isomorphic in codimension $1$.

\section{Proof of the main theorem}\label{sec3}
In this section, we prove Theorem \ref{main-thm}. The proof splits into two parts: the klt case and the non-klt case. We will use different methods to treat them.

\subsection{Klt case}

In this subsection, we deal with the klt case.

\begin{thm}\label{thm-klt}
Let $m$ be a positive integer. 
Then there exists a positive integer $N$ depending only on $m$ satisfying the following property:
if $(S, B)$ is a projective connected
klt pair of dimension $2$ such that $m(K_S+B)\sim 0$, then
$|\rho_{km}(\Bir(S,B))|\leq N$ for any positive integer $k$. 
\end{thm}

\begin{proof}
First we consider the case that $B=0$ and $S$ has
at worst du Val singularities. In this case, the boundedness of $|\rho_m(\Bir(S,B))|$ is well-known to experts (cf. \cite[Proposition 3.6]{fujino-index}). We briefly recall the proof here for the reader's convenience. 
 By Remark \ref{rem bir=bir}, we may assume that $S$ is smooth after taking the minimal resolution. 
 Take $r$ to be the minimal positive integer such that $rK_S\sim 0$. Take $\tilde S \to S$ to be the index 1 cover of $K_S$, that is, the $r$-fold cyclic covering given by a non-zero section 
in $H^0(S, rK_S)$, then $\tilde S$ is a projective smooth surface with $K_{\tilde S}\sim 0$. By the classification theory of surfaces (for example \cite[Chapter VIII]{Be}), $b_2(\tilde S)\leq 22$. Hence by Lemma \ref{cover-lem}, there exists a constant $N_1$ independent of $S$ such that $|\rho_r(\Bir(S,B))|\leq N_1$. In particular, 
$|\rho_{m}(\Bir(S,B))|\leq |\rho_r(\Bir(S,B))|\leq N_1$ as $r$ divides $m$.

From now on, we consider the case that $B\neq 0$ or $S$ has
worst than du Val singularities.
In this case, $S$ belongs to a bounded family by \cite[Theorem 6.9]{alexeev}. 
Moreover, as $m(K_S+B)\sim 0$, $(S,B)$ belongs to a log bounded family by standard arguments (see, for example, \cite[Lemma 2.20]{Bir}).
Then by Theorem \ref{bdd family} below, there exist two positive integers $k_2$ and $N_2$ independent of $(S, B)$ such that 
 $k_2(K_S+B)\sim 0$ and
$|\rho_{k_2}(\Bir(S,B))|\leq N_2$. By Remark \ref{remark1}, $|\rho_{m}(\Bir(X,\Delta))|\leq k_2 |\rho_{k_2m}(\Bir(X,\Delta))|\leq k_2N_2$.

So summarizing two cases, we can just take $N=\max\{N_1, k_2N_2\}$.
\end{proof}

We show the following more general result on boundedness of B-pluricanonical representations in a family which is log bounded modulo B-birational contractions.

\begin{thm}\label{bdd family}
Let $m, d$ be two positive integers.
Let $\mathcal{D}$ be a set of connected klt log Calabi--Yau pairs of dimension $d$ which is log bounded modulo B-birational contractions, such that for each $(X, B)\in \mathcal{D}$, $mB$ is an integral divisor. 
Then there exist two positive integers $k$ and $N$ depending only on $\mathcal{D}$ such that,
if $(X, B)\in \mathcal{D}$, then $k(K_X+B)\sim 0$ and $|\rho_k(\Bir(X,B))|\leq N$.
\end{thm}

\begin{proof}
Note that if $(X, B)\dashrightarrow (X', B')$ is a B-birational contraction and $(X, B)$ is a connected klt log Calabi--Yau pair, then $(X', B')$ is automatically a connected klt log Calabi--Yau pair and $mB'$ is integral. Moreover,
 $k(K_X+B)\sim 0$ if and only if $k(K_{X'}+B')\sim 0$, and in this case $|\rho_k(\Bir(X,B))|=|\rho_k(\Bir(X',B'))|$ by Remark \ref{rem bir=bir}. So after replacing $\mathcal{D}$ by $\mathcal{D}'$ as in the definition of log boundedness modulo B-birational contractions, we may further assume that $\mathcal{D}$ is a log bounded family of connected klt log Calabi--Yau pairs of dimension $d$.
 Note that by Global ACC \cite[Theorem 1.5]{HMX14} (see \cite[Proof of Proposition 3.1]{hacon-xu}), there exists a constant $\epsilon \in (0,1)$ such that $(X, B)$ is $\epsilon$-lc for any $(X, B)\in \mathcal{D}$.
 
 By the definition of log boundedness, there is a  scheme $\mathcal{X}$ and a projective morphism $h: \mathcal{X}\to T$, a 
reduced divisor $\mathcal{B}'$ on $\mathcal X$, where 
$T$ is of finite type and $\mathcal{B}'$ does not contain 
any fiber, such that for every $(X,B)\in \mathcal{D}$, 
there is a closed point $t \in T$ and an isomorphism
 $f \colon \mathcal{X}_t \to X$ 
such that $\mathcal{B}'_t:=\mathcal{B}'|_{\mathcal{X}_t}$ 
coincides with the support of $f_*^{-1}B$. As the coefficients of $B$ are in a fixed finite set, after replacing $T$ by disjoint union of locally closed subsets, we may assume that there exists a 
$\mathbb Q$-divisor $\mathcal{B}$ on $\mathcal X$ such that for every $(X,B)\in \mathcal{D}$, $(X, B)\simeq (\mathcal{X}, \mathcal{B})|_{\mathcal{X}_t}$ for some fiber ${\mathcal{X}_t}$. Moreover, by applying \cite[Proposition 2.4]{hacon-xu} and the Noetherian induction, we may further assume that $K_{\mathcal X}+\mathcal{B}$ is $\mathbb{Q}$-Cartier, $ ({\mathcal X}, \mathcal{B})$ is klt, and the set of points $t$ corresponding to $(X,B)\in \mathcal{D}$ is dense in $T$.
Replacing $T$ by disjoint union of locally closed subsets $\cup_i T_i$ while taking log resolutions of $\mathcal{X}$, we may assume that there are finitely many smooth varieties $T_i$ and projective morphisms $(\mathcal{W}_i, \mathcal{E}_i)\to (\mathcal{X}_i, \mathcal{B}_i)\to T_i$ such that $(\mathcal{W}_i, \mathcal{E}_i)$ is log smooth over $T_i$ and for every $t\in T_i$, the fiber $\mathcal{X}_{i, t}$ is a normal projective variety of dimension $d$,
 $(\mathcal{W}_{i,t}, \mathcal{E}_{i, t})$ is a log resolution of $(\mathcal{X}_{i, t}, \mathcal{B}_{i,t})$ with $\mathcal{E}_{i, t}$ the sum of strict transform of $\mathcal{B}_{i,t}$ and the
reduced exceptional divisor, and the set of points $t$ corresponding to $(X,B)\in \mathcal{D}$ is dense in each $T_i$. 

Note that if $(X, B)\in \mathcal{D}$ is isomorphic to a fiber $(\mathcal{X}_{i, t_0}, \mathcal{B}_{i,t_0})$ of $(\mathcal{X}_i,\mathcal{B}_i) \to T_i$, then it is a good minimal model of $(\mathcal{W}_{i,t_0}, \mathcal{E}_{i, t_0})$. Hence by \cite[Corollary 1.4]{HMX18}, for each positive integer $l$ such that $l\mathcal{B}_{i,t}$ is integral,
 $$
h^0( \mathcal{X}_{i, t}, l(K_{ \mathcal{X}_{i, t}}+\mathcal{B}_{i,t}))=h^0(\mathcal{W}_{i,t}, l(K_{\mathcal{W}_{i,t}}+\mathcal{E}_{i, t}))
$$
is constant for $t\in T_i$. Since the set of points $t$ corresponding to $(X,B)\in \mathcal{D}$ is dense in each $T_i$, over the generic point $\eta_i\in T_i$,
 $K_{\mathcal{X}_{i, \eta_i}}+ {\mathcal{B}}_{i, \eta_i}\sim_{\mathbb{Q}} 0$. So by the Noetherian induction, further replacing $T$ by disjoint union of locally closed affine subsets (still denoted by $\cup_i T_i$), we may assume that $K_{\mathcal{X}}+ \mathcal{B}\sim_{\mathbb{Q}} 0$.

As there are only finitely many $T_i$, we can reduce to the case that $T=T_i$. Note that by the construction, every fiber $(\mathcal{X}_{t}, \mathcal{B}_{t})$ is a connected klt log Calabi--Yau pair. Recall that by \cite[Corollary 1.4]{HMX18}, for each positive integer $l$ such that $l\mathcal{B}_{t}$ is integral (this condition is independent of $t$),
$
h^0( \mathcal{X}_{t}, l(K_{ \mathcal{X}_{t}}+\mathcal{B}_{t}))$
is constant for $t\in T$, which implies that the index of $(\mathcal{X}_{t}, \mathcal{B}_{t})$ is a constant positive integer, denoted by $k$. In particular, $k$ is also the minimal positive integer such that $k(K_{\mathcal{X}}+ \mathcal{B})\sim 0$, as $T$ is affine.

Consider the log resolution $\psi: ({\mathcal{W}}, \Delta)\to ({\mathcal{X}}, \mathcal{B})$ which is a log resolution on each fiber, where $K_{\mathcal{W}}+ \Delta=\psi^*(K_{\mathcal{X}}+ \mathcal{B})$. Then for a non-zero section $\omega\in H^0( \mathcal{W}, k(K_{\mathcal{W}}+ \Delta))$, we can consider $\mu: {\tilde{\mathcal{W}}}_\omega\to \mathcal{W}$ to be the $k$-fold cyclic covering given by the 
section $\omega$ and take $\phi: \mathcal{V} \to ({\tilde{\mathcal{W}}}_\omega,\Delta_{{\tilde{\mathcal{W}}}_\omega})$ to be a log resolution. 
Here we may assume that on each fiber of $t\in T$, 
 $\phi_t: \mathcal{V}_t
\to ({\tilde{\mathcal{W}}}_{\omega,t},\Delta_{{\tilde{\mathcal{W}}}_{\omega, t}})$ is a log resolution by the Noetherian induction after replacing $T$ by disjoint union of locally closed affine subsets.
Since $\mathcal{V}_t$ is in a bounded family, $b_d(\mathcal{V}_t)$ has a uniform upper bound. Hence by Lemma \ref{cover-lem}, there exists a constant $N$ such that for each $t\in T$, $|\rho_k(\Bir(\mathcal{X}_t, \mathcal{B}_t))|=|\rho_k(\Bir(\mathcal{W}_t, \Delta_t))|\leq N $.
\end{proof}

Besides the proof of Theorem \ref{thm-klt}, Theorem \ref{bdd family} has also the following applications to other known bounded families of log Calabi--Yau pairs.

\begin{cor}
Let $d$ be a positive integer and $I$ a finite set in $[0,1]\cap \mathbb{Q}$.
Then there exist two positive integers $k$ and $N$ depending only on $d, I$ satisfying the following property: if $(X, B) $ is a connected klt log Calabi--Yau pair of dimension $d$, $B$ is big and
 the coefficients of $B$ are in $I$,
then $k(K_X+B)\sim 0$ and $|\rho_k(\Bir(X,B))|\leq N$.
\end{cor}
\begin{proof}
It follows directly from \cite[Theorem 1.3]{hacon-xu} and Theorem {\ref{bdd family}}.
\end{proof}

\begin{cor}
Let $I$ be a finite set in $[0,1]\cap \mathbb{Q}$.
Then there exist two positive integers $k$ and $N$ depending only on $I$ satisfying the following property: if $(X, B) $ is a connected klt log Calabi--Yau pair of dimension $3$ such that $X$ is rationally connected and
 the coefficients of $B$ are in $I$,
then $k(K_X+B)\sim 0$ and $|\rho_k(\Bir(X,B))|\leq N$.
\end{cor}
\begin{proof}
By Theorem {\ref{bdd family}}, it suffices to show that such $(X, B)$ belongs to a log bounded family modulo B-birational contractions.
In fact, such $(X, B)$ belongs to a log bounded family modulo flops, which is proved by \cite[Theorem 4.1]{rccy3} if $B\neq 0$ and by \cite[Corollary 5.5]{rccy3} and \cite[Theorem 1.3]{jiang} if $B= 0$.
\end{proof}

\begin{cor}
There exist two positive integers $k$ and $N$ satisfying the following property: if $X $ is a
non-canonical klt Calabi--Yau variety of dimension $3$, then $kK_X\sim 0$ and $|\rho_k(\Bir(X,0))|\leq N$.
\end{cor}
\begin{proof}
By Theorem {\ref{bdd family}}, it suffices to show that such $(X, 0)$ belongs to a log bounded family modulo B-birational contractions.
In fact, such $(X, 0)$ belongs to a log bounded family modulo flops, which is proved by \cite[Theorem 1.4]{jiang}.
\end{proof}

\subsection{Non-klt case}

In this subsection, we deal with the non-klt case.

\begin{thm}\label{thm-dlt}
Let $m$ be a positive integer. 
Then there exists a positive integer $N$ depending only on $m$ satisfying the following property:
if $(S, B)$ is a projective connected non-klt
lc pair of dimension $2$ such that $m(K_S+B)\sim 0$, then
$|\rho_{km}(\Bir(S,B))|\leq N$ for any positive integer $k$. 
\end{thm}

\begin{rem}
In the proof of Theorem \ref{thm-dlt}, we follow the idea in \cite[Proof of Theorem 3.15, Case 1]{fujino-gongyo}. But as we are interested in the boundedness of B-pluricanonical representations, we need to put extra effort. For example, we need to show the lc centers of $(S, B)$ satisfy certain boundedness. To this end, we show in Lemma \ref{lem-good dlt model 2} that certain special lc centers of $(S, B)$ are bounded after replacing $(S, B)$ by a B-birational model. Then we can modify the argument in \cite[Proof of Theorem 3.15, Case 1]{fujino-gongyo} to prove 
Theorem \ref{thm-dlt}. On the other hand, in Section \ref{sec4}, following the idea in \cite[Theorem 3.5]{fujino-ab} and \cite[Page 560, Step 2]{gongyo}, we can give an alternative proof of Theorem \ref{thm-dlt} which involves more technical inductive arguments by using (pre-)admissible sections, see Remark \ref{remark 2nd proof thm-dlt}.
\end{rem}

Firstly we recall the following well-known connectedness lemma 
(cf. \cite[12.3.2]{kollar92}, 
\cite[Proposition 2.1]{fujino-ab}, \cite[Proposition 2.4]{fujino-index},
\cite[Proposition 5.1]{kk}, \cite[Claim 5.3]{gongyo}).

\begin{lem}[{Connectedness lemma, \cite[Claim 5.3]{gongyo}}]\label{lem-num-1}
Let $(X,\Delta)$ be a projective connected lc pair with $K_X+\Delta \sim_{\mathbb{Q}} 0$.
Then $\lfloor\Delta \rfloor$ has at most 2 connected components.
\end{lem}

Recall that a set of smooth curves is a \emph{chain} (resp. \emph{cycle}) if their dual graph is a chain (resp. cycle), and \emph{end curves} of a chain are those curves corresponding to end points of the dual graph. Here we allow two curves intersecting at two distinct points to be a cycle.
\begin{lem}\label{lem-chain or loop}
Let $(S, B)$ be a projective connected dlt pair of dimension $2$ such that $K_S+B \sim_{\mathbb Q} 0$. Then $\lfloor B \rfloor$ has at most 2 connected components, and each connected component is either a chain or a cycle. In this setting, we usually write $\lfloor B \rfloor=T_0+T_1$, where $T_0$ consists of connected components which are cycles, and $T_1$ consists of connected components which are chains, and denote $T_e$ to be the sum of all end curves in $T_1$.
\end{lem}
 
\begin{proof}
As $(S, B)$ is dlt, all irreducible components of $\lfloor B \rfloor$ are normal curves and any two 
curves intersect transversally at a smooth point of $S$. Moreover, for any irreducible component $D$ of $\lfloor B \rfloor$, we have $K_D+B_D=(K_S+B)|_D\sim_{\mathbb Q} 0$ by the adjunction formula. Hence $D$ is either a rational curve or an elliptic curve, and $\deg B_D\leq 2$, which implies that each irreducible component of $\lfloor B \rfloor$ intersects at most two other irreducible components of $\lfloor B \rfloor$. 
 \end{proof}
\begin{lem}\label{lem-log resol chain}
Let $(S, B)$ be a projective connected dlt pair of dimension $2$ such that $K_S+B \sim_{\mathbb Q} 0$. 
Let $\pi: S' \to S$ be any log resolution of $(S, B)$, write $\pi^*(K_S+B)=K_{S'}+B'$.
Then the image of $\pi$ gives a natural 1-1 correspondence 
between connected components of $B'^{=1}$ and those of $\lfloor B \rfloor$, which preserves the type of each connected component. 
Moreover, if one connected component of $B'^{=1}$ is a chain, then $\pi$ maps the end curves of this connected component to those of the corresponding connected component of $\lfloor B \rfloor$ .
\end{lem}

\begin{proof}
After replacing $S$ by its minimal resolution, we may assume that $S$ is smooth. Note that this does not change $\lfloor B \rfloor$. 
We can factor $\pi: S' \to S$ into blowups at points
$$
S'=S_k\to S_{k-1}\to\cdots S_1\to S_0=S,
$$
 where $\pi_i: S_i \to S_{i-1}$ is a blowup at a smooth point, and write $\pi_i^*(K_{S_{i-1}}+B_{i-1})=K_{S_i}+B_i$.
 Then we can consider $B_i^{=1}$ in each step.
 If the blowed up point is not a $0$-dimensional stratum of $B_{i-1}^{=1}$, then
 $B_i^{=1}$ is isomorphic to $B_{i-1}^{=1}$; if the blowed up point is a $0$-dimensional stratum of $B_{i-1}^{=1}$, then 
 $B_i^{=1}$ is obtained by replacing this point by a smooth rational curve, and in this case, this blowup preserves types of irreducible components and end curves of chains. So we can conclude the lemma inductively.
\end{proof}

The following lemma is well-known to experts (see, for example, \cite[Lemma 1.4]{alexeev-mori}).

\begin{lem}\label{lem-number}
Let $S$ be a smooth projective minimal surface and $B=\sum_i b_iB_i$ an effective $\mathbb Q$-divisor on $S$ such that $b_i\leq 1$ for all $i$. Assume that $K_S+B \sim_{\mathbb Q} 0$.
Then one of the following is true:
\begin{enumerate}
\item $B=0$ and $K_S\sim_{\mathbb Q} 0$; 
\item $S\simeq \mathbb P^2$ with $\sum_i b_i\leq 3$;
\item $S\simeq \mathbb F_n$ for some integer $n\geq 2$ with $\sum_i b_i\leq 4$;
\item $S$ is a $\mathbb P^1$-bundle over an elliptic curve with $\sum_i b_i\leq 2$.
\end{enumerate}
In particular, $\lfloor B \rfloor$ has at most $4$ irreducible components.
\end{lem}

\begin{proof}
If $B=0$, then $K_S\sim_{\mathbb Q} 0$, which gives (1).
From now on, we assume that $B\neq 0$. Then $K_S\sim_{\mathbb Q} -B\neq 0$ is not pseudo-effective. 
By the minimality of $S$, $S$ is either $\mathbb P^2$ or a $\mathbb P^1$-bundle over a smooth curve $C$.

Suppose that $S\simeq \mathbb P^2$. Then $\sum_i b_i\leq \sum_i b_i(B_i\cdot \ell)= (-K_S\cdot \ell)=3$ by $K_S+B \equiv 0$, where $\ell$ is a line on $\mathbb P^2$. This is (2).

Suppose that $f: S\to C$ is a $\mathbb P^1$-bundle over a smooth curve $C$. Since $K_S+B\sim_{\mathbb Q} 0$, by the canonical
bundle formula (see \cite[Theorem 3.1]{fg-cbd}), $-K_C$ is pseudo-effective, which implies that $C$ is either a rational curve or an elliptic curve.
If $C$ is a rational curve, then $S\simeq \mathbb F_n$ and $\sum_i b_i\leq 4$
by \cite[Lemma 1.3]{alexeev-mori}. This is (3).
If $C$ is an elliptic curve, then we know that $K_S+B\sim_{\mathbb Q} f^*K_C$. Again from the the canonical
bundle formula, $B$ does not contain any fiber of $f$. Hence $\sum_i b_i\leq \sum_i b_i(B_i\cdot F)= (-K_S\cdot F)=2$, where $F$ is a fiber of $S\to C$. 
\end{proof}

\begin{lem}\label{lem-good dlt model 1}
Let $(S, B)$ be a projective connected lc pair of dimension $2$ such that $K_S+B \sim_{\mathbb Q} 0$. Suppose that $(S, \lfloor B \rfloor)$ is log smooth.
Then there exists a projective birational morphism $\pi: (S', B')\to (S, B )$, where $K_{S'}+B'=\pi^*(K_{S}+B)$, such that $S'$ is smooth, $(S', B')$ is dlt, and if we write $\lfloor B' \rfloor=T_0+T_1$ as in Lemma \ref{lem-chain or loop}, then $T_0\leq \pi_*^{-1} \lfloor B \rfloor$.
\end{lem}

\begin{proof}
As $(S, \lfloor B \rfloor)$ is log smooth, we can construct a dlt model of $(S, B )$ by a sequence of blowups along $0$-dimensional lc centers of $(S, B)$ avoiding blowing up 0-dimensional strata of {$\lfloor B \rfloor$}, say $\pi: (S', B')\to (S, B )$, where $K_{S'}+B'=\pi^*(K_{S}+B)$, such that $S'$ is smooth and $(S', B')$ is dlt. Note that irreducible components of $\lfloor B' \rfloor$ consist of irreducible components of $\pi_*^{-1} \lfloor B \rfloor$ and some exceptional divisors of $\pi$. By the construction of $S'$, there is no exceptional divisor appearing in $T_0$, so $T_0\leq \pi_*^{-1} \lfloor B \rfloor$.
\end{proof}

The following lemma is the key lemma in this subsection, which tells that for any connected lc log Calabi--Yau pair of dimension $2$, there is a ``nice" B-birational model.
\begin{lem}\label{lem-good dlt model 2}
Let $(S, B)$ be a projective connected lc pair of dimension $2$ such that $K_S+B \sim_{\mathbb Q} 0$. 
Then there exists a dlt pair $(S', B')$ B-birational to $(S, B)$ such that $S'$ is smooth, and if we write $\lfloor B' \rfloor=T_0+T_1$ as in Lemma \ref{lem-chain or loop}, then $T_0$ has at most $6$ irreducible components.
\end{lem}

\begin{proof}
We may assume that $S$ is smooth after taking the minimal resolution. After running a $K_S$-MMP, we get a birational {morphism} $\pi: S\to S_0$ to a minimal surface $S_0$. Then $K_{S_0}+\pi_*B \sim_{\mathbb Q} 0$ and $\lfloor \pi_*B \rfloor$ has at most $4$ irreducible components by Lemma \ref{lem-number}. Note that $(S_0, \pi_*B)$ is lc but not necessarily dlt. $(S_0, \pi_*B)$ being lc implies that any two curves in $\lfloor \pi_*B \rfloor$ intersect transversally.

For a component $D$ of $\lfloor \pi_*B \rfloor$, we know that 
$$
2p_a(D)-2={(K_{S_0}+D)\cdot D\leq (K_{S_0}+\pi_*B)\cdot D} =0.
$$ 
Hence $p_a(D)=0$ or $1$, moreover, if $p_a(D)=1$, then $D$ is disjoint from $\Supp(\pi_*B- D)$. As  $\lfloor \pi_*B \rfloor$ has at most 2 connected components, it has at most 2 singular irreducible components. Moreover, a singular irreducible component can be resolved by the blowup at its singular point. So after at most two blowups, we get a B-birational model $(S_1, B_1)\to (S_0, \pi_*B)$ such that $(S_1, \lfloor B_1\rfloor)$ is log smooth and $\lfloor B_1\rfloor$ has at most $6$ irreducible components. Then we can apply Lemma \ref{lem-good dlt model 1} to 
$(S_1, B_1)$.
\end{proof}

\begin{rem}
The bound on the number of irreducible components of $T_0$ could be much sharper by carefully discussions case by case. Indeed, it can be shown that we can construct $(S', B')$ such that $T_0$ has at most $4$ irreducible components.
Since we won't need a sharp bound in this paper, we are satisfied with
the bound in Lemma \ref{lem-good dlt model 2} and left the details to those interested readers.
\end{rem}

The following lemma is a modification of \cite[Remark 2.15]{fujino-gongyo} in our situation.
\begin{lem}\label{lem-induced auto}
Let $(S, B)$ be a projective connected dlt pair of dimension $2$ such that $m(K_S+B) \sim 0$
for a positive integer $m$. Consider $\lfloor B \rfloor =T_0+T_1$ and $T_e$ as in Lemma \ref{lem-chain or loop}. Assume that irreducible components of $T_e$ are disjoint from each other.
Then a B-birational map 
$g: (S,B) \dashrightarrow (S,B) $
induces natural automorphisms
$$
g^*: H^0(T_0, m(K_{ T_0 }+B_{ T_0 }))
\to H^0( T_0 ,m(K_{ T_0 }+B_{ T_0 })),
$$
$$
g^*: H^0(T_e, m(K_{ T_e }+B_{ T_e }))
\to H^0( T_e, m(K_{ T_e }+B_{ T_e }))
$$
where $K_{ T_0 }+B_{T_0}=(K_S+B)|_{ T_0 }$, $K_{ T_e }+B_{T_e}=(K_S+B)|_{ T_e }$. 
\end{lem}

\begin{proof}
Consider a common log resolution
$$
\xymatrix{
&(S', B')\ar[ld]_{\alpha}\ar[rd]^{\beta}\\ 
(S,B) \ar@{-->}[rr]^{g} & & (S,B),
}
$$
where $K_{S'}+B'=\alpha^*(K_S+B)=\beta^*(K_S+B)$.
By Lemma \ref{lem-log resol chain}, we can consider $B'^{=1}=T'_0+T'_1$ and $T'_e$ accordingly to $\lfloor B\rfloor=T_0+T_1$ and $T_e$.
Note that this expression is independent of $\alpha$ and $\beta$. Since irreducible components of $T_e$ are disjoint from each other, it is clear that $\alpha_*\mathcal{O}_{T'_e}=T_e$ and $\beta_*\mathcal{O}_{T'_e}=T_e$.
By \cite[Remark 2.15]{fujino-gongyo}, we have $\alpha_*\mathcal{O}_{T'_0}=T_0$ and $\beta_*\mathcal{O}_{T'_0}=T_0$.
Hence we get natural automorphisms
\begin{align*}
g^*: H^0(T_0, m(K_{ T_0 }+B_{ T_0 }))
{}&\stackrel{\beta^*}{\longrightarrow} H^0( T'_0 , m(K_{ T'_0 }+B_{ T'_0 })) \\
{}&\stackrel{(\alpha^{*})^{-1}}{\longrightarrow} 
 H^0( T_0 , m(K_{ T_0 }+B_{ T_0 })),
\end{align*}
\begin{align*}
g^*: H^0(T_e, m(K_{ T_e }+B_{ T_e }))
{}&\stackrel{\beta^*}{\longrightarrow} H^0( T'_e , m(K_{ T'_e }+B_{ T'_e })) \\
{}&\stackrel{(\alpha^{*})^{-1}}{\longrightarrow} 
 H^0( T_e , m(K_{ T_e }+B_{ T_e })),
\end{align*}
where $K_{ T'_0 }+B_{ T'_0 }=(K_{S'}+{B'})|_{ T'_0 }$ and $K_{ T'_e }+B_{ T'_e }=(K_{S'}+{B'})|_{ T'_e }$.
\end{proof}

The following lemma is a modification of \cite[Lemma 2.16]{fujino-gongyo} in our situation.

\begin{lem}\label{lem-induced b map}
Let $(S, B)$ be a projective connected dlt pair of dimension $2$ such that $K_S+B \sim_{\mathbb Q} 0$. Consider $\lfloor B \rfloor =T_0+T_1$ and $T_e$ as in Lemma \ref{lem-chain or loop}. 
Let 
$g: (S,B) \dashrightarrow (S,B) $ be a B-birational map.
Take a common log resolution
$$
\xymatrix{
&(S', B')\ar[ld]_{\alpha}\ar[rd]^{\beta}\\ 
(S,B) \ar@{-->}[rr]^{g} & & (S,B),
}
$$
where $K_{S'}+B'=\alpha^*(K_S+B)=\beta^*(K_S+B)$.
By Lemma \ref{lem-log resol chain}, we can consider $B'^{=1}=T'_0+T'_1$ and $T'_e$ accordingly to $\lfloor B\rfloor=T_0+T_1$ and $T_e$.
\begin{enumerate}
\item For any irreducible component $C$ of $T_e$, there exists an irreducible component $C'$ of $T_e$ and an irreducible component $D$ of $T'_e$, such that $ \alpha|_D$ and $ \beta|_D$ are B-birational morphisms
$$
\xymatrix{
&(D, B'_D)\ar[ld]_{\alpha|_D}\ar[rd]^{\beta|_D}\\ 
(C,B_C) \ar@{-->}[rr] & & (C',B_{C'}),
}
$$
where $K_C+B_C=(K_S+B)|_C$, $K_{C'}+B_{C'}=(K_S+B)|_{C'}$, and $K_D+ B'_D=(K_{S'}+B')|_D$.
Therefore, $ \beta|_D\circ (\alpha|_D)^{-1}: (C,B_C) \dashrightarrow (C',B_{C'})$ is a B-birational map.

\item For any lc center $P$ of $(S, B)$ in $T_0$, we can find an lc center $Q$ of $(S, B)$ contained in $P$, an lc center $R$ of $(S', B')$, and an lc center $Q'$ of $(S, B)$ contained in $T_0$ such that $ \alpha|_R$ and $ \beta|_R$ are B-birational morphisms
$$
\xymatrix{
&(R, B'_R)\ar[ld]_{\alpha|_R}\ar[rd]^{\beta|_R}\\ 
(Q,B_Q) \ar@{-->}[rr] & & (Q',B_{Q'}),
}
$$
where $K_Q+B_Q=(K_S+B)|_Q$, $K_{Q'}+B_{Q'}=(K_S+B)|_{Q'}$, and $K_R+ B'_R=(K_{S'}+B')|_R$.
Therefore, $ \beta|_R\circ (\alpha|_R)^{-1}: (Q,B_Q) \dashrightarrow (Q',B_{Q'})$ is a B-birational map.
Moreover, by the natural restriction map, $H^0(P, m(K_P+B_P))\simeq H^0(Q, m(K_Q+B_Q))$ where $K_P+B_P=(K_S+B)|_P$.
\end{enumerate}
\end{lem}

\begin{proof}
(1) is a direct consequence of Lemma \ref{lem-log resol chain}. (2) is just \cite[Lemma 2.16]{fujino-gongyo}, the fact that $Q'$ is in $T_0$ follows from Lemma \ref{lem-log resol chain}.
\end{proof}

Now we are ready to prove Theorem \ref{thm-dlt}.

\begin{proof}[Proof of Theorem \ref{thm-dlt}]
By Remark \ref{rem bir=bir},
we can replace $(S,B)$ by a B-birational dlt model as in Lemma \ref{lem-good dlt model 2}.
In particular, we can assume that $(S, B)$ is dlt, $S$ is smooth, and $T_0$ has at most $6$ irreducible components.
Here we consider $\lfloor B \rfloor =T_0+T_1$ and $T_e$ as in Lemma \ref{lem-chain or loop}. Note that $\lfloor B \rfloor \neq 0$ as $(S, B)$ is not klt. Fix any $g\in \Bir(S,B)$.

Suppose that $T_1\neq 0$. 
We can consider $T_e$. As $T_1$ has at most 2 connected components, ${T_e}$ has at most 4 irreducible components.
After further blowups, we may assume that irreducible components of $T_e$ are disjoint from each other.
Since $m(K_S+B)\sim 0$, $$
H^0(S, m(K_S+B)-{T_e})=H^0(S, -{T_e})=0.
$$
Therefore, the restriction map
$$
H^0(S, m(K_S+B)) \to H^0({T_e}, m(K_{T_e}+B_{T_e}))
$$
is injective, where $K_{T_e}+B_{T_e}=(K_S+B)|_{T_e}$.
Write $T_e=\amalg_{i=1}^kC_i$ with $k\leq 4$. Write $K_{C_{i}}+ B_{C_{i}}=(K_S+B)|_{C_i}$. Applying Lemma \ref{lem-induced b map}(1), for each index $i$, $g$ induces a B-birational map $(C_i, B_{C_i})\dashrightarrow (C_{i'}, B_{C_{i'}})$ for some index $i'$. Therefore, $g^{k!}$ induces a B-birational map $(C_i, B_{C_i})\dashrightarrow (C_{i}, B_{C_{i}})$ for each $i$.
Note that $m(K_{C_{i}}+ B_{C_{i}})\sim 0$ and $C_i$ is of dimension 1, hence
by \cite[Proposition 8.4]{fmx19},
there exists an integer $N_1$ depending only on $m$ such that
$$
|\rho_m(\Bir({C_i}, B_{C_i}))|\leq N_1.
$$ 
Let $h:=24\cdot (N_1)!$.
Then $(g^*)^h$ induces the identity map on $H^0({C_i}, m(K_{C_i}+B_{C_i}))$ for each $i$.
By Lemma \ref{lem-induced auto}, there is a 
commutative diagram
$$
\xymatrix{
0\ar[r] &H^0(S, m(K_S+B)) \ar[d]_{(g^*)^h}\ar[r]& 
H^0({T_e}, m(K_{T_e}+B_{T_e}))\ar[d]^{(g^*)^h=\id}\\ 
0\ar[r] & H^0(S, m(K_S+B)) \ar[r] & H^0({T_e}, m(K_{T_e}+B_{T_e})).
}
$$
It follows that $(g^*)^h=\id$ on $H^0(S, m(K_S+B))$. In particular,
 $|\rho_m(\Bir(S,B))|\leq h$.

Suppose that $T_0\neq 0$. Then by construction, $T_0$ has at most $6$ irreducible components
 and at most $6$ 0-dimensional strata.
In particular, $T_0$ contains at most $12$ lc centers of $(S, B)$, denoted by $P_1, \ldots, P_l$ for $l\leq 12$.
Since $m(K_S+B)\sim 0$, 
$$
H^0(S, m(K_S+B)-{T_0})=H^0(S, -{T_0})=0.
$$
Therefore, the restriction map
$$
H^0(S, m(K_S+B)) \to H^0({T_0}, m(K_{T_0}+B_{T_0}))
$$
is injective, where $K_{T_0}+B_{T_0}=(K_S+B)|_{T_0}$. 
By Lemma \ref{lem-induced b map}(2),
for every $P_i$, we can find lc centers $Q_i, Q'_i$ of $(S,B)$ in $T_0$ such that $(Q_i, B_{Q_i}) \dashrightarrow ( Q'_i, B_{Q'_i})$
is a B-birational map and 
$$
H^0(P_i, m(K_{P_i}+B_{P_i})) \simeq 
H^0(Q_i, m(K_{Q_i}+B_{Q_i}))
$$
by the natural restriction map, where $K_{P_i}+B_{P_i}=(K_S+B)|_{P_i}$
and $K_{Q_i}+B_{Q_i}=(K_S+B)|_{Q_i}$. Here $Q_i, Q'_i$ also belong to the set $\{P_1, \ldots, P_l\}$.
Therefore, $g^{l!}$ induces a natural B-birational map $g^{l!}: (Q_i, B_{Q_i}) \dashrightarrow (Q_i, B_{Q_i})$ for each $i$. 
Note that $m(K_{Q_{i}}+ B_{Q_{i}})\sim 0$ and $Q_i$ is of dimension at most 1, hence
by \cite[Proposition 8.4]{fmx19},
there exists an integer $N_1$ depending only on $m$ such that
$$
|\rho_m(\Bir({Q_i}, B_{Q_i}))|\leq N_1.
$$ 
Let $h':=12!\cdot (N_1)!$.
Then $(g^*)^{h'}$ induces the identity map on $H^0({Q_i}, m(K_{Q_i}+B_{Q_i}))$ for each $i$.
By the natural embedding 
$$
H^0(T_0, m(K_{T_0}+B_{T_0})) \subset
\bigoplus_i H^0(P_i, m(K_{P_i}+B_{P_i}))\simeq \bigoplus_i H^0(Q_i, m(K_{Q_i}+B_{Q_i})),
$$
 $(g^*)^{h'}$ induces the identity map on $H^0(T_0, m(K_{T_0}+B_{T_0}))$.
By Lemma \ref{lem-induced auto}, there is a 
commutative diagram
$$
\xymatrix{
0\ar[r] &H^0(S, m(K_S+B)) \ar[d]_{(g^*)^{h'}}\ar[r]& 
H^0({T_0}, m(K_{T_0}+B_{T_0}))\ar[d]^{(g^*)^{h'}=\id}\\ 
0\ar[r] & H^0(S, m(K_S+B)) \ar[r] & H^0({T_0}, m(K_{T_0}+B_{T_0})).
}
$$
It follows that $(g^*)^{h'}=\id$ on $H^0(S, m(K_S+B))$. In particular,
 $|\rho_m(\Bir(S,B))|\leq h'$.
 \end{proof}

\section{Applications}\label{sec4}

In this section, we discuss applications of Theorem \ref{main-thm} to Conjecture \ref{index conj slc}. Note that Conjecture \ref{index conj slc} can be viewed as the effective version of the abundance conjecture for log Calabi--Yau pairs, so the framework of Fujino \cite{fujino-ab} provides an inductive argument between Conjecture \ref{main-conj} and Conjecture \ref{index conj slc}.

\subsection{Fujino's work on (pre-)admissible sections}
In this subsection, we recall some key ideas of \cite{fujino-ab} (see also \cite{gongyo, xu2}).
 {{Pre-admissible sections}} and {{admissible sections}} are used in
 the inductive proof of the index conjecture.

\begin{defn}[{\cite[Definition 4.1]{fujino-ab}}]\label{admin}
Let $(X,\Delta)$ be a projective sdlt pair of dimension $d$. 
Let $m$ be a positive integer such that $m(K_X+\Delta)$ is Cartier. 
Let $\nu: X'=\amalg_i X'_i \to X=\cup_i X_i$ be the normalization.
Let $\Delta'$ be the $\mathbb Q$-divisor such that $K_{X'}+\Delta'=\nu^*(K_X+\Delta)$
and $\Delta'_i=\Delta'|_{X'_i}$.
We define {\em{pre-admissible section}} and {\em{admissible section}} inductively on dimension:
\begin{itemize}
\item[$(1)$] $s\in H^0(X, m(K_X+\Delta))$ is {\em{pre-admissible}} 
if the restriction 
$$
\nu^* s|_{(\amalg_i \lfloor \Delta'_i\rfloor)} \in 
H^0(\amalg_i \lfloor \Delta'_i\rfloor, m(K_{X'}+\Delta')|_{(\amalg_i \lfloor \Delta'_i\rfloor)})
$$ 
is admissible. Denote the set of pre-admissible sections by $\PA (X, m(K_X+\Delta))$.
\item[$(2)$] $s\in H^0(X, m(K_X+\Delta))$ is {\em{admissible}} if $s$ is pre-admissible
and $g^*(s|_{X_j})=s|_{X_i}$ for every B-birational map 
$$
g: (X_i, \Delta_i) \dashrightarrow (X_j, \Delta_j)
$$
for every $i, j$. Denote the set of admissible sections by $\A (X, m(K_X+\Delta))$.
\end{itemize}
Note that if $s\in \A(X, m(K_X+\Delta))$, 
then $s|_{X_i}$ is $\Bir(X_i, \Delta_i)$-invariant for every $i$.
\end{defn}

We can run the same inductive argument as in
 \cite[Section 4]{fujino-ab} (see also \cite[Section 5]{gongyo} and \cite[Section 5]{xu1}). In the following we briefly recall the key results with {proofs} following their ideas.
 Taking boundedness into account, Theorems A, B, C in \cite{gongyo} can be formulated into the following conjectures:

\begin{conja}
Let $m, d$ be two positive integers. Then there exists 
a positive integer $M$ depending only on $m, d$ satisfying the following property:
if $(X,\Delta)$ is a projective (not necessarily connected) dlt pair
of dimension $d$ such that $m(K_X+\Delta)\sim 0$, 
then $\PA(X, Mm(K_X+\Delta))\neq 0$.
\end{conja}

\begin{conjb}[{=Conjecture \ref{main-conj}}]
Let $m, d$ be two positive integers. Then there exists 
a positive integer $N$ depending only on $m, d$ satisfying the following property:
if $(X, \Delta)$ is a projective connected
dlt pair of dimension $d$ such that $m(K_X+\Delta)\sim 0$, then
$|\rho_m(\Bir(X,\Delta))|\leq N$. 
\end{conjb}

\begin{conjb'}
Let $m, d$ be two positive integers. Then there exists 
a positive integer $N$ depending only on $m, d$ satisfying the following property:
if $(X, \Delta)$ is a projective connected
klt pair of dimension $d$ such that $m(K_X+\Delta)\sim 0$, then
$|\rho_m(\Bir(X,\Delta))|\leq N$. 
\end{conjb'}

\begin{conjc}
Let $m, d$ be two positive integers. Then there exists 
a positive integer $M$ depending only on $m, d$ satisfying the following property:
if $(X,\Delta)$ is a projective (not necessarily connected) dlt pair
of dimension $d$ such that $m(K_X+\Delta)\sim 0$, 
then $\A(X, Mm(K_X+\Delta))\neq 0$.
\end{conjc}

In the following, Conjecture $\bullet_d$ (resp. Conjecture $\bullet_{\leq d}$) stands for Conjecture $\bullet$ with $\dim X = d$ (resp. $\dim X \leq d$). Note that for the above conjectures, Conjecture $\bullet_d$ naturally implies Conjecture $\bullet_{d-1}$ by considering fiber products with an elliptic curve.

The following propositions are 
taken out from \cite[Section 5]{gongyo} and \cite[Section 4]{fujino-ab}
by minor modifications of proofs.

\begin{prop}[{cf. \cite[Proposition 4.5]{fujino-ab}, \cite[Proposition 4.15]{fujino-index}, \cite[Claim 5.4]{gongyo}, \cite[Lemma 5.11]{xu1}}]\label{C to A}
If Conjecture C$_{d-1}$ holds true, then Conjecture A$_{d}$ holds true.
\end{prop}
\begin{proof}
Let $(X,\Delta)$ be a projective dlt pair
of dimension $d$ such that $m(K_X+\Delta)\sim 0$.
Write $(X, \Delta)=\amalg_i (X_i , \Delta_i)$ where $(X_i , \Delta_i)$ is a projective connected dlt pair of dimension $d$ for each $i$.

If $\lfloor \Delta \rfloor= 0$, then this is trivial because any section in $H^0(X, m(K_{X}+\Delta))$ is pre-admissible. 

If $\lfloor \Delta \rfloor\neq 0$, then by Conjecture C$_{d-1}$, there exists a positive integer $M$ depending only on $m, d$ such that $\A(\lfloor \Delta \rfloor, Mm(K_{X}+\Delta)|_{\lfloor \Delta \rfloor})\neq 0$. We may assume that $M$ is even. Then by \cite[Proposition 4.5]{fujino-ab}, for each $i$, $$\PA(X_i, Mm(K_{X_i}+\Delta_i)))\to \A(\lfloor \Delta_i \rfloor, Mm(K_{X_i}+\Delta_i)|_{\lfloor \Delta_i \rfloor})$$ is surjective. Here we remark that \cite[Proposition 4.5]{fujino-ab} requires that $\dim X_i\leq 3$ and $Mm$ is sufficiently divisible, where the former condition is for applying \cite[Proposition 2.1]{fujino-ab} which can be removed by using \cite[Claim 5.3]{gongyo}, and the latter condition can be replace by $Mm(K_{X_i}+\Delta_i)\sim 0$ and $Mm$ is even, hence we can apply \cite[Proposition 4.5]{fujino-ab} in our situation, see also \cite[Lemma 5.11]{xu1} for detailed {discussions}.
As
$$
\A(\lfloor \Delta \rfloor, Mm(K_{X}+\Delta)|_{\lfloor \Delta \rfloor})\subset \bigoplus_i \A(\lfloor \Delta_i \rfloor, Mm(K_{X_i}+\Delta_i)|_{\lfloor \Delta_i \rfloor}),
$$
 there exists a non-zero section $t\in H^0(X, Mm(K_{X}+\Delta))$ such that $t|_{\lfloor \Delta \rfloor}\in\A(\lfloor \Delta \rfloor, Mm(K_{X}+\Delta)|_{\lfloor \Delta \rfloor})$. By definition, $t$ is pre-admissible, and hence $\PA(X, Mm(K_X+\Delta))\neq 0$.
\end{proof}

\begin{prop}[{cf. \cite[Theorem 3.5]{fujino-ab}, \cite[Page 560, Step 2]{gongyo}}]\label{AB to B}
If Conjecture A$_d$ and Conjecture B'$_{d}$ hold true, then Conjecture B$_d$ holds true.
\end{prop}

\begin{proof}
 Let $(X, \Delta)$ be a projective connected
dlt pair of dimension $d$ such that $m(K_X+\Delta)\sim 0$.

If $(X, \Delta)$ is klt, then by Conjecture B'$_{d}$, there exists a positive integer $N$ depending only on $m,d$ such that $|\rho_m(\Bir(X,\Delta))|\leq N$.

If $(X, \Delta)$ is not klt, then there exists a non-zero section $t\in \PA(X, m(K_X+\Delta))$
by Conjecture A$_d$. Note that in this case 
 $\PA(X, m(K_X+\Delta))=H^0(X, m(K_X+\Delta))\simeq \mathbb{C}$. So $\rho_m(g)\in \mathbb{C}^*$ for any $g\in \Bir(X, \Delta)$. On the other hand, by \cite[Proposition 4.9]{fujino-ab}, for any $g\in \Bir(X, \Delta)$, $g^*t|_{\lfloor \Delta \rfloor}=t|_{\lfloor \Delta \rfloor}$. So $\rho_m(g)=1$ for any $g\in \Bir(X, \Delta)$, that is, $|\rho_m(\Bir(X,\Delta))|= 1$.
\end{proof}

\begin{prop}[{cf. \cite[Lemma 4.9]{fujino-ab}, \cite[Proposition 5.10]{xu1}}]\label{AB to C}
If Conjecture A$_d$ and Conjecture B'$_{d}$ hold true, then 
 Conjecture C$_{d}$ holds true.
\end{prop}

\begin{proof}

Let $(X,\Delta)$ be a projective (not necessarily connected) dlt pair
of dimension $d$ such that $m(K_X+\Delta)\sim 0$.
Write $(X, \Delta)=\amalg_i (X_i , \Delta_i)$ where $(X_i , \Delta_i)$ is a projective connected dlt pair of dimension $d$ for each $i$.
Our goal is to construct a non-zero section in $\A(X, km(K_X+\Delta))$ for some positive integer $k$ independent of $(X,\Delta)$.

We may write $(X, \Delta)=(X' , \Delta')\amalg (X'' , \Delta'')$ into two parts where $(X' , \Delta')$ consists of all klt $(X_i , \Delta_i)$ and {$(X'' , \Delta'')$} consists of all non-klt $(X_i , \Delta_i)$. Then $$\A(X, km(K_{X}+\Delta))=\A(X', km(K_{X'}+\Delta'))\oplus \A(X'', km(K_{X''}+\Delta''))$$ 
for all positive integer $k$ because there is no B-birational map between a klt pair and a non-klt pair.
Hence we only need to consider 2 cases: $(X_i , \Delta_i)$ is klt for each $i$, or $(X_i , \Delta_i)$ is not klt for each $i$.

Suppose that $(X_i , \Delta_i)$ is not klt for each $i$. By Conjecture A$_d$, there exists a positive integer $M$ depending only on $m,d$ such that $\PA(X, Mm(K_X+\Delta))\neq 0$. 
Then we claim that $\A(X, Mm(K_X+\Delta))=\PA(X, Mm(K_X+\Delta))\neq 0$. In fact, 
it suffices to show that if $g:(X_i, \Delta_i) \dashrightarrow (X_j, \Delta_j)$ is a B-birational map and $s\in \PA(X, Mm(K_X+\Delta))$, then $g^*(s|_{X_j})=s|_{X_i}$. As $H^0(X_i, Mm(K_{X_i}+\Delta_i))\simeq H^0(X_j, Mm(K_{X_j}+\Delta_j))\simeq \mathbb{C}$, we may assume that $g^*(s|_{X_j})=\lambda s|_{X_i}$ for some $\lambda\in \mathbb{C}^*$. On the other hand, $g$ induces a natural B-birational map $\tilde{g}\in \Bir(X, \Delta)$ by exchanging $X_i$ and $X_j$, hence by \cite[Lemma 4.9]{fujino-ab}, $g^*(s|_{\lfloor{\Delta_j}\rfloor})= s|_{\lfloor{\Delta_i}\rfloor}$. 
So either $\lambda=1$ or $s|_{\lfloor{\Delta_i}\rfloor}=0$. If $s|_{\lfloor{\Delta_i}\rfloor}=0$, then $s|_{\lfloor{\Delta_j}\rfloor}=0$ and hence $g^*(s|_{X_j})= s|_{X_i}=0$. So in both cases $g^*(s|_{X_j})= s|_{X_i}$.

If $(X_i , \Delta_i)$ is klt for each $i$, then $$\PA(X, m(K_X+\Delta))=H^0(X, m(K_X+\Delta)) = \bigoplus_i H^0(X, m(K_{X_i}+\Delta_i)).$$
Fix a section $(s_i)_i \in H^0(X, m(K_X+\Delta)) $ where $s_i \in H^0(X, m(K_{X_i}+\Delta_i)) $ is non-zero for each $i$.
 By Conjecture B'$_{d}$, there exists a positive integer $N$ depending only on $m,d$ such that $|\rho_m(\Bir(X_i,\Delta_i))|\leq N$.
This implies that $\rho_{N!m}(\Bir(X_i,\Delta_i))$ acts trivially on $H^0(X, N!m(K_{X_i}+\Delta_i))\simeq \mathbb{C}$ for each $i$. Denote $t=(t_i)_i=(s_i^{N!})_i \in H^0(X, N!m(K_X+\Delta)) $.
If $g_{ij}: (X_i, \Delta_i)\dashrightarrow (X_j, \Delta_j)$ is a B-birational map, then we can write $g_{ij}^*t_j=\lambda_{ij}t_i$ for some $\lambda_{ij}\in \mathbb{C}^*$. Note that if $f_{ij}: (X_i, \Delta_i)\dashrightarrow (X_j, \Delta_j)$ is another B-birational map, then $f_{ij}^{-1}\circ g_{ij}\in \Bir(X_i,\Delta_i)$. As $\rho_{N!m}(\Bir(X_i,\Delta_i))$ acts trivially on $H^0(X, N!m(K_{X_i}+\Delta_i))$, this implies that  $\lambda_{ij}$ is independent of the choice of $g_{ij}$  and $\lambda_{ii}=1$ for each $i$. So we can find $\lambda_i\in \mathbb{C}^*$ for each $i$ such that $\lambda_{ij}=\lambda_i\lambda_j^{-1}$ if there is a B-birational map $(X_i, \Delta_i)\dashrightarrow (X_j, \Delta_j)$.
Then it is easy to check that the non-zero section $t'=(\lambda_i t_i)_i \in H^0(X, N!m(K_X+\Delta)) =\PA(X, N!m(K_X+\Delta))$ is actually admissible.
\end{proof}

From the above inductive arguments, it is easy to get the following corollary.
\begin{cor}\label{cor C2A3}
Conjecture C$_{\leq 2}$ and Conjecture A$_{\leq 3}$ hold ture. 
\end{cor} 
\begin{proof}
This directly follows from Propositions \ref{C to A} and \ref{AB to C} as Conjecture B'$_{\leq 2}$ holds true by Theorem \ref{thm-klt}. 
\end{proof}
\begin{rem}\label{remark 2nd proof thm-dlt}
By Proposition \ref{AB to B} and Corollary \ref{cor C2A3}, Conjecture B$_{\leq 2}$ holds true. Note that here we only use Theorem \ref{thm-klt}, so this indeed gives an alternative proof of Theorem \ref{thm-dlt}.
\end{rem}

\subsection{Applications to the index conjecture}
In this subsection, we give applications of Theorem \ref{main-thm} to the index conjecture.

The following propositions are well-known to experts as inductive steps towards the index conjecture.
\begin{prop}[{cf. \cite[Theorem 1.5]{gongyo}}]\label{prop dlt to slc}
If Conjecture \ref{index conj slc} holds true for dlt pairs of dimension $d$ and Conjecture A$_d$ holds, then Conjecture \ref{index conj slc} holds true in dimension $d$. 
\end{prop}

\begin{proof}The proof is essentially the same as \cite[Theorem 1.5]{gongyo}.
Let $(X, \Delta)$ be a projective slc pair of dimension $d$ such that the coefficients of $\Delta$ are in $I$
and $K_X+\Delta\sim_{\mathbb Q} 0$. We may assume that $X$ is connected.
Take the normalization $X'=\amalg_i X'_i \to X = \cup_i X_i$ and a dlt blowup (\cite[Theorem 3.1]{kk}) on each $X'_i$. We get $\phi: (Y,\Gamma) \to (X,\Delta)$ such that $(Y, \Gamma)$ is a projective (not necessarily connected) dlt pair and $K_Y + \Gamma = \phi^*(K_X + \Delta)$. Note that $K_Y + \Gamma \sim_\mathbb{Q} 0$ and the coefficients of $\Gamma$ are in $I\cup \{1\}$, hence by Conjecture \ref{index conj slc} for dlt pairs of dimension $d$, there exists a positive integer $m$ depending only on $d, I$ such that $m(K_Y+\Gamma)\sim 0$. By Conjecture A$_d$, after replacing $m$ by a constant multiple, we may assume that $m(K_Y+\Gamma)$ has a non-zero pre-admissible section, which descends to a non-zero section of $m(K_X+\Delta)$ by \cite[Lemma 4.2]{fujino-ab}. Hence $m(K_X+\Delta)\sim 0$.
\end{proof}

\begin{prop}[{cf. \cite[Theorem 1.7]{xu2}}]\label{prop slc to nonklt}
If Conjecture \ref{index conj slc} holds true in dimension $d-1$, then Conjecture \ref{index conj slc} holds true for connected non-klt lc pairs in dimension $d$. 
\end{prop}

Here we give a simple proof different from \cite[Theorem 1.7]{xu2}, which follows the idea of \cite[9.9]{PS}.
\begin{proof}
Let $(X, \Delta)$ be a projective connected non-klt lc pair of dimension $d$ such that the coefficients of $\Delta$ are in $I$
and $K_X+\Delta\sim_{\mathbb Q} 0$. 
After taking a dlt blowup (\cite[Theorem 3.1]{kk}) and replacing $I$ by $I\cup \{1\}$, we may assume that $(X, \Delta)$ is dlt and in particular $\lfloor \Delta\rfloor \neq 0$.

 If we write $(K_X+\Delta)|_{\lfloor \Delta \rfloor}=K_{\lfloor \Delta \rfloor}+\Theta$, then $({\lfloor \Delta \rfloor}, \Theta)$ is an sdlt pair of dimension $d-1$ by \cite[Remark 1.2(3)]{fujino-ab} such that $K_{\lfloor \Delta \rfloor}+\Theta\sim_{\mathbb Q} 0$. By the adjunction formula (\cite[Corollary 16.7]{kollar92} or \cite[Lemma 4.1]{HMX14}), the coefficients of $\Theta$ belong to a set $D(I)$ which is a DCC (i.e., descending chain condition) set of rational numbers depending only on $I$. Then by the Global ACC (\cite[Theorem 1.5]{HMX14}), the coefficients of $\Theta$ belong to a finite set $I_0$ of rational numbers depending only on $I$. Hence
by Conjecture \ref{index conj slc} in dimension $d-1$, there exists a positive integer $m$ depending only on $d, I_0$ such that $m(K_X+\Delta)|_{\lfloor \Delta \rfloor} \sim 0$. Note that $m$ depends only on $d, I$ as $I_0$ depends only on $I$.
As $I$ is a finite set, after replacing $m$ by a multiple, we may also assume that $m(K_X+\Delta)$ is a Weil divisor. 
Take the minimal positive integer $r$ such that $rm(K_X+\Delta)\sim 0$.
Take the index 1 cover of $m(K_X+\Delta)$, say $\pi: Y \to X$, then $K_Y+\Delta_Y=\pi^*(K_X+\Delta) \sim_\mathbb{Q} 0$ where $\Delta_Y$ is an effective $\mathbb Q$-divisor since $\pi$ is \'etale in codimension one. Then $(Y, \Delta_Y)$ is a connected lc pair of dimension $d$ by \cite[Proposition 5.20]{kollar-mori}. As $m(K_X+\Delta)|_{\lfloor \Delta \rfloor} \sim 0$, $\lfloor\Delta_Y\rfloor=\pi^{-1}\lfloor\Delta\rfloor$ has at least $r$ connected components. Hence by the connectedness lemma (Lemma \ref{lem-num-1}), $r\leq 2$. In particular, $2m(K_X+\Delta)\sim 0$.
\end{proof}

Now we can get some partial results of the index conjecture in low dimensions. 

\begin{proof}[Proof of Corollary \ref{ind-threefold}]
It suffices to proof the corollary for $\dim X=3$, as lower dimensional case can be reduced to higher dimensional case by taking fiber products with an elliptic curve.
By Proposition \ref{prop dlt to slc} and Corollary \ref{cor C2A3}, it suffices to prove Conjecture \ref{index conj slc} for dlt pairs in dimension $3$.
If $X$ is klt and $\Delta=0$, then this is proved by \cite[Corollary 1.7]{jiang}.
The remaining cases are proved by \cite[Theorem 1.13]{xu2}.
\end{proof}

\begin{proof}[Proof of Corollary \ref{ind-fourfold}]
This follows directly from Proposition \ref{prop slc to nonklt} and Corollary \ref{ind-threefold}.
\end{proof}


\end{document}